\documentclass[11pt,twoside]{amsart}
\usepackage{amsmath, amsthm, amscd, amsfonts, amssymb, graphicx, color, alltt}
\usepackage[bookmarksnumbered, plainpages]{hyperref}
\usepackage{alltt}

\newtheorem{thm}{Theorem}[section]
\newtheorem{cor}[thm]{Corollary}
\newtheorem{lem}[thm]{Lemma}

\newtheorem{exa}[thm]{Example}
\newtheorem{conj}[thm]{Conjecture}
\newtheorem{rem}[thm]{\bf{Remark}}

\numberwithin{equation}{section}
\usepackage[utf8]{inputenc}
\usepackage{pstricks-add}

\title{\bf Counting the Number of   Centralizers of $2-$Element Subsets in a Finite Group}

\author{\bf  A. R. Ashrafi$^\star$, F. Koorepazan-Moftakhar and M. A. Salahshour}

\thanks{$^\star$Corresponding author (Email: ashrafi@kashanu.ac.ir)}

\address{Department of Pure Mathematics, Faculty of Mathematical Sciences,
University of Kashan, Kashan 87317$-$53153, I. R. Iran}

\date{}
\begin{document}

\maketitle

\begin{abstract}
Suppose $G$ is a finite group. The set of all centralizers of $2-$element subsets of $G$ is denoted by $2-Cent(G)$. A group $G$ is called $(2,n)-$centralizer if $|2-Cent(G)| = n$ and primitive $(2,n)-$centralizer if $|2-Cent(G)| = |2-Cent(\frac{G}{Z(G)})| = n$, where $Z(G)$ denotes the center of $G$. The aim of this paper is to present the main properties of $(2,n)-$centralizer groups among them a characterization of $(2,n)-$centralizer and primitive $(2,n)-$centralizer groups, $n \leq 9$, are given.

\vskip 3mm


\vskip 3mm

\end{abstract}

\bigskip

\section{Introduction}
Throughout this paper all groups are assumed to be finite and for a subset $A$ of a group $G$, the centralizer subgroup of $A$ in $G$ is denoted by $C_G(A)$. The cyclic group of order $n$, the dihedral group of order $2n$, the alternating group on $n$ symbols and the symmetric group of degree $n$ are denoted by  $\mathbb{Z}_n$, $D_{2n}$, $A_n$ and $S_n$, respectively. The holomorph of $G$ is denoted by $Hol(G)$ and $Z(G)$ denotes the center of $G$. The second center of $G$ which is denoted by $Z_2(G)$, is defined  as  $\frac{Z_2(G)}{Z(G)} = Z(\frac{G}{Z(G)})$. Set:
\begin{eqnarray*}
Cent(G) &=& \{ C_G(x) \mid x \in G\},\\
2-Cent(G) &=& \{ C_G(\{ x, y\}) \mid x, y \in G \ \& \ x \ne y\}.
\end{eqnarray*}
The group $G$ is called $n-$centralizer if $|Cent(G)| = n$ and in addition, if $\frac{G}{Z(G)}$  is  an $n-$centralizer group, then $G$ is said to be primitive $n-$centralizer group. It is called $(2,n)-$centralizer if $|2-Cent(G)| = n$ and primitive $(2,n)-$centralizer if $|2-Cent(G)| = |2-Cent(\frac{G}{Z(G)})| = n$.  To simplify our argument, a subgroup of $G$ in the form $C_G(\{ x, y\}) $ is called a $2-$element centralizer for $G$. It is clear that the number of $2-$element centralizers of $G$ is equal to $|2-Cent(G)|$. Note that for each element $x, y \in G$, $C_G(\{x,y\}) = C_G(x) \cap C_G(y)$.

The study of finite groups in terms of $|Cent(G)|$ was started by Belcastro and Sherman in 1994 \cite{6}. It is easy to see  that a group is $1-$centralizer if and only if it is abelian and there is no $2-$ and $3-$centralizer group. One of the present authors (ARA) \cite{2} constructed
$n-$centralizer groups, for each $n \ne 2, 3$ and primitive $n-$centralizer groups for each odd positive integer $\ne 3$. In \cite{4}, the authors proved that a group with $\leq 21$ element centralizers is solvable and gave a characterization of the simple group $A_5$ according to the number of element centralizers. Zarrin \cite{14} presented another proof for solvability of groups with at most $21$ element centralizers. Zarrin \cite{13} computed $|Cent(G)|$ for all minimal simple groups. As a consequence he proved that there are non-isomorphic finite simple groups
$G$ and $H$ such  that $|Cent(G)| = |Cent(H)|$. Kitture \cite{9} proved that  isoclinic groups have the same number of element centralizers and for each positive integer $n$ different from $2$ and $3$, there are only finitely many groups, up to isoclinism, with exactly $n$ element centralizers.

Suppose $R$ denotes the semidirect product of a cyclic group of order $5$ by a cyclic group of order $4$ acting faithfully. It is easy to see that
$R$ can be presented as  $R = \langle x, y \mid x^5 = y^4 = 1, xy = yx^3\rangle$. A finite $n-$centralizer group $G$ with $n \leq 10$ were determined in terms of the structure of $\frac{G}{Z(G)}$. For the sake of completeness, we collect these results in the following theorems:

\begin{thm}\label{thm1.1}
Suppose $G$ is a finite group. Then,
\begin{enumerate}
\item  $G$ is $4-$centralizer if and only if $\frac{G}{Z(G)} \cong \mathbb{Z}_2 \times \mathbb{Z}_2$ \cite{6}.

\item  $G$ is $5-$centralizer if and only if $\frac{G}{Z(G)} \cong \mathbb{Z}_3 \times \mathbb{Z}_3$ or $S_3$ \cite{6}.

\item If $G$ is $6-$centralizer, then $\frac{G}{Z(G)} \cong \mathbb{Z}_2 \times \mathbb{Z}_2 \times \mathbb{Z}_2$, $\mathbb{Z}_2 \times \mathbb{Z}_2 \times \mathbb{Z}_2 \times \mathbb{Z}_2$, $D_8$ or $A_4$ \cite{2}.

\item $G$ is primitive $7-$centralizer if and only if $\frac{G}{Z(G)} \cong  D_{10}$ or $R$ \cite{3}.

\item $G$ is $7-$centralizer if and only if $\frac{G}{Z(G)} \cong \mathbb{Z}_5 \times \mathbb{Z}_5$, $D_{10}$ or $R$ \cite{1}.

\item If $G$ is $8-$centralizer, then $\frac{G}{Z(G)} \cong \mathbb{Z}_2 \times \mathbb{Z}_2 \times \mathbb{Z}_2$, $D_{12}$ or $A_4$ \cite{1}.

\item $G$ is $9-$centralizer if and only if $\frac{G}{Z(G)} \cong D_{14}$, $\mathbb{Z}_7 \times \mathbb{Z}_7$, $Hol(\mathbb{Z}_7)$ or a non-abelian group of order 21 \cite{8}.

\item  $G$ is primitive $9-$centralizer if and only if $\frac{G}{Z(G)} \cong D_{14}$,  $Hol(\mathbb{Z}_7)$ or a non-abelian group of order $21$ \cite{5}.

\item  There is no $10-$centralizer groups of odd order \cite{7}.

\item If $G$ is a primitive $11-$centralizer group of odd order, then $\frac{G}{Z(G)} \cong (\mathbb{Z}_9 \times \mathbb{Z}_3) \rtimes \mathbb{Z}_3$ \cite{10}.
\end{enumerate}
\end{thm}

A group in which every non-central element has an abelian centralizer is called a $CA-$group \cite{115}.

\begin{thm}\label{thm1.2}
Suppose $n$ is a positive integer and $p$ is a prime. Then we have:
\begin{enumerate}

\item $($\cite[Lemma 2.4.]{1}$)$ Let $G$ be a finite non-abelian group and $\{x_1, \ldots , x_r\}$ be a set of
pairwise non-commuting elements of $G$ with maximal size. Then \label{thm1.2(1)}
\begin{enumerate}
\item $r \geq 3$. \label{thm1.2(1a)}
\item $r +1 \leq |Cent(G)|$. \label{thm1.2(1b)}
\item $r = 3$ if and only if $|Cent(G)| = 4$. \label{thm1.2(1c)}
\item $r = 4$ if and only if $|Cent(G)| = 5$. \label{thm1.2(1d)}
\end{enumerate}

\item $($\cite[Proposition 2.5]{1}$)$ Let $G$ be a finite group and let $X = \{x_1, \ldots, x_r\}$ be a set of
pairwise non-commuting elements of $G$ having maximal size. \label{thm1.2(2)}
\begin{enumerate}
\item If $|Cent(G)| < r + 4$, then for each element $x \in G$, $C_G(x)$ is abelian if and only if $C_G(x) = C_G(x_i)$
for some $i \in \{1, \ldots, r\}$. \label{thm1.2(2a)}

\item If $|Cent(G)| = r + 2$, then there exists a proper non-abelian centralizer
$C_G(x)$ which contains $C_G(x_{i_1})$, $C_G(x_{i_2} )$ and $C_G(x_{i_3})$ for three distinct $i_1$, $i_2$, $i_3 \in \{1, \ldots, r\}$. \label{thm1.2(2b)}
\end{enumerate}

\item $($\cite[Lemma 2.6]{1}$)$ Let $G$ be a finite non-abelian group. Then every proper centralizer
of $G$ is abelian if and only if $|Cent(G)| = r + 1$, where $r$ is the maximal size of a
set of pairwise non-commuting elements of $G$. \label{thm1.2(3)}

\item $($\cite[Theorem 1]{4}$)$ If $G$ is a finite group and $\frac{G}{Z(G)} \cong A_5$, then $|Cent(G)| = 22$ or $32$.\label{thm1.2(new)}

\item $($\cite[Theorem 2]{4}$)$ If $G$ is a finite simple group and $|Cent(G)| = 22$, then $G \cong A_5$.\label{thm1.2(new1)}

\item $($\cite[Lemma 2.1]{55}$)$ Let $|\frac{G}{Z(G)}| = pqr$, where $p, q$ and $r$ are primes not necessarily distinct. Then $G$ is $CA-$group. \label{thm1.2(4)}

\item $($\cite[Proposition 2.2]{55}$)$ Let $p$ be the smallest prime dividing $|G|$. If $[G : Z(G)] = p^3$,
then $|Cent(G)| = p^2 + p + 2$ or $p^2 + 2$. \label{thm1.2(45)}

\item $($\cite[Theorem 2.3]{55}$)$ If $G$ has an abelian normal subgroup of prime index, then $|Cent(G)| = |G'| + 2$. \label{thm1.2(5)}

\item $($\cite[Proposition 2.8]{55}$)$ Let $\frac{G}{Z(G)}$ be non-abelian, $n$ be an integer and $p$ be a prime. If $\frac{G}{Z(G)} \cong \mathbb{Z}_n \rtimes \mathbb{Z}_p$,  then $G$ has an abelian normal subgroup of index $p$ and $|G'| = n$. \label{thm1.2(6)}

\item $($\cite[Proposition 2.9]{55}$)$ Let $\frac{G}{Z(G)}$ be non-abelian. If $\frac{G}{Z(G)} \cong \mathbb{Z}_n \rtimes \mathbb{Z}_p$, then $|Cent(G)| = n + 2$. \label{thm1.2(7)}

\item $($\cite[Lemma 2.10]{55}$)$ If $\frac{G}{Z(G)}$ is  non-abelian and $\frac{G}{Z(G)} \cong \mathbb{Z}_n \rtimes \mathbb{Z}_p$ then
$G$ is a $CA-$group. \label{thm1.2(8)}

\item $($\cite[Theorem 5]{6}$)$ Let $p$ ba a prime. If $\frac{G}{Z(G)} \cong \mathbb{Z}_p \times \mathbb{Z}_p$, then $|Cent(G)| = p+2$. \label{thm1.2(9)}

\item $($\cite[Lemma 4, p. 303]{66}$)$ If $H$ is a normal abelian subgroup of a non-abelian group $G$ of prime index $p$, then $|G| = p|G'||Z(G)|$, and $|G : C_G(x)| = |G'|$ for $x \in G - H$. \label{thm1.2(10)}

\item $($\cite[Theorem A (I)]{65}$)$ Let $G$ be a non-abelian group. Then $G$ is a $CA-$group if and only if $G$ has an abelian normal subgroup of prime index. \label{thm1.2(11)}

\item $($\cite[Proposition 1.2]{Ito}$)$ If $G$ is a $CA-$group, then for each non-central element $x$ and $y$, $C_G(x) = C_G(y)$ or $C_G(x) \cap C_G(y) = Z(G)$. \label{thm1.2(12)}

\item $($\cite{4} and \cite[Theorem A]{14}$)$ Let $G$ be an $n-$centralizer finite group with $n \leq 21$, then $G$ is soluble. \label{thm1.2(15)}
\end{enumerate}
\end{thm}

Throughout this paper our notations are standard and taken mainly from \cite{11}. Our calculations are done with the aid of GAP \cite{12}.

\section{Some Basic Properties of $(2,n)-$Centralizer Groups }
In this section, a characterization of finite groups with at most nine  $2-$element  centralizers are given.

\begin{lem}\label{lemm2.1}
Let $G$ be a non-abelian group. Then
\begin{enumerate}
\item $Z(G) \notin Cent(G)$.
\item $G \notin 2-Cent(G)$ if and only if $Z(G) = 1$.
\end{enumerate}
\end{lem}

\begin{proof}
Let $Z(G) \in Cent(G)$. Set $Cent(G) = \{G, C_G(x_1), \ldots, C_G(x_r)\}$, where $x_i \notin Z(G)$, $1 \leq i \leq r$. Then either $Z(G) = G$ or there exists $x_i$, $1 \leq i \leq r$, such that $Z(G) = C_G(x_i)$. Both of these cases will lead to contradiction. Thus, $Z(G) \notin Cent(G)$ that  completes Part (1). To prove the Part (2), it is enough to note that for every $x, y \in G$ such that $x \neq y$, $x, y \in Z(G)$ if and only if $C_G(x) \cap C_G(y) = G$.
\end{proof}

The previous lemma shows that if $Z(G) \in Cent(G)$, then $G$ is abelian. The following simple lemmas are crucial in our main results.

\begin{lem}\label{lemm2.2}
Let $G$ be a non-abelian group. Then
\begin{enumerate}
\item If $Z(G) \neq 1$, then $Cent(G) \varsubsetneqq 2-Cent(G)$. In particular $|Cent(G)| < |2-Cent(G)|$.
\item If $Z(G) = 1$, then $|Cent(G)| \leq |2-Cent(G)|$.
\end{enumerate}
\end{lem}

\begin{proof}
Our main proof will consider the following two separate cases:
\begin{enumerate}
\item Let $Z(G) \neq 1$ and $Cent(G) = \{G, C_G(x_1), \ldots, C_G(x_r)\}$ in which $x_i \notin Z(G)$, $1 \leq i \leq r$. Since $Z(G) \neq 1$, by Lemma \ref{lemm2.1}(2), $G \in 2-Cent(G)$. On the other hand, for every $1 \leq i \leq r$, $C_G(x_i) = C_G(\{x_i, e\}) \in 2-Cent(G)$. Therefore, $Cent(G) \subseteq 2-Cent(G)$. Suppose $2-Cent(G) = Cent(G)$. Then for every $x, y \in G \setminus Z(G)$
$$C_G(x) \cap C_G(y) = C_G(\{x, y\}) \in 2-Cent(G) = Cent(G) = \{G, C_G(x_1), \ldots, C_G(x_r)\}$$
and by a similar argument as the proof of Lemma \ref{lemm2.1}(2), $C_G(x) \cap C_G(y) \neq G$. So there exists $1 \leq i \leq r$, such that $C_G(x) \cap C_G(y) = C_G(x_i)$. By our assumption, $Z(G) = C_G(x_1) \cap C_G(x_2) \cap \ldots \cap C_G(x_r)$. Now an inductive argument shows that $Z(G) = C_G(x_t)$, for some $t$, $1 \leq t \leq r$. Therefore, $x_t \in Z(G)$ which is impossible.

\item Suppose $Z(G) = 1$ and set $Cent(G) = \{G, C_G(x_1), \ldots, C_G(x_r)\}$ in which $x_i \notin Z(G)$, $1 \leq i \leq r$. By Lemma \ref{lemm2.1}(2), $G \notin 2-Cent(G)$ and according to Lemma \ref{lemm2.2}(1), $Cent(G) \setminus \{G\} \subset 2-Cent(G)$. Thus $|Cent(G)| \leq |2-Cent(G)|$.
\end{enumerate}
This completes the proof.
\end{proof}

\begin{lem}\label{lemm2.3}
Suppose $H$ and $K$ are two finite groups. Then $|Cent(H \times K)| = |Cent(H)||Cent(K)|$. In particular, if $H$ is abelian, then $|Cent(H \times K)| = |Cent(K)|$.
\end{lem}

\begin{proof}
It is clear that if $(x, y) \in H \times K$, then $C_{H \times K}((x,y)) = C_H(x) \times C_K(y)$. We now consider four separate cases for the pair $(x,y)$ as follows:
\begin{enumerate}
\item \textit{$x$ is a central element of $H$ and $y$ is a central element of $K$.} In this case, $C_{H \times K}((x,y)) = C_H(x) \times C_K(y) = H \times K$. Therefore in this case  we have only one element centralizer in $H \times K$.
\item \textit{$x$ is not central in $H$ but $y$ is a central element of $K$.} Since $C_{H \times K}((x,y)) = C_H(x) \times K$, there are exactly $|Cent(H)| - 1$ element centralizers different from $H \times K$.
\item \textit{$x$ is a central element of $H$ and $y$ is not central in $K$.} In this case, $C_{H \times K}((x,y)) = H \times C_K(y)$ and again we have $|Cent(K)| - 1$ element centralizers different from $H \times K$. Note that all of these element centralizers are different from those given in part (2).
\item \textit{$x$ is not central in $H$ and $y$ is not central in $K$.} By our assumption, $C_{H \times K}((x,y)) = C_H(x) \times C_K(y)$ and there are $(|Cent(H)| -1)(|Cent(K)|-1)$ different element centralizers in $H \times K$.
\end{enumerate}
Therefore
\begin{eqnarray*}
|Cent(H \times K)| &=& 1 + (|Cent(H)| - 1) + (|Cent(K)| - 1) +  (|Cent(H)| - 1))(|Cent(K)| - 1)\\
&=& |Cent(H)||Cent(K)|.
\end{eqnarray*}
The second part is obvious.
\end{proof}

\begin{thm}\label{thm2.4}
Suppose $H$ and $K$ are two finite groups. Then
$$|2-Cent(H \times K)| = |2-Cent(H)||2-Cent(K)| + \delta(K)|2-Cent(H)| + \delta(H)|2-Cent(K)|$$
in which
$$
\delta(G) = \left\lbrace
\begin{array}{cc}
1 & Z(G) = 1 \\
0 & Z(G) \neq 1
\end{array}
 \right..
$$
\end{thm}

\begin{proof}
It is clear that for every $(x,y) \in H \times K$, $C_{H \times K}((x,y)) = C_H(x) \times C_K(y)$. For every $(m,n) \in H \times K$, $(m,n) \in C_{H \times K}((x,y)) \cap C_{H \times K}((a,b))$ if and only if $(m,n) \in C_{H \times K}((x,y))$ and $(m,n) \in C_{H \times K}((a,b))$ if and only if $(m,n) \in (C_H(x) \times C_K(y))$ and $(m,n) \in (C_H(a) \times C_K(b))$ if and only if $m \in C_H(x)$, $n \in C_K(y)$, $m \in C_H(a)$ and $n \in C_K(b)$ if and only if $m \in C_H(x) \cap C_H(a)$ and $n \in C_K(y) \cap C_K(b)$ if and only if $(m,n) \in \left(C_H(x) \cap C_H(a) \right) \times \left(C_K(y) \cap C_K(b) \right)$. Therefore, for every $(x,y), (a,b) \in H \times K$,
\begin{equation}\label{eq1}
C_{H \times K}((x,y)) \cap C_{H \times K}((a,b)) = \left(C_H(x) \cap C_H(a) \right) \times \left(C_K(y) \cap C_K(b) \right).
\end{equation}

We have the following four cases:

\begin{enumerate}
\item \textit{$Z(H) \neq 1$ and $Z(K) \neq 1$}. By Lemma \ref{lemm2.2}(1),  $Cent(H) \subset 2-Cent(H)$ and $Cent(K) \subset 2-Cent(K)$. Now by Equation \ref{eq1},  $|2-Cent(H \times K)|$ = $|2-Cent(H)||2-Cent(K)|.$

\item \textit{$Z(H) \neq 1$ and $Z(K) = 1$}. Again apply Lemma \ref{lemm2.2}(1) to show that $Cent(H) \subset 2-Cent(H)$ and by  Lemma \ref{lemm2.1}(2) $K \in Cent(K) \setminus 2-Cent(K)$. Therefore, $$|2-Cent(H \times K)| = |2-Cent(H)||2-Cent(K)| + |2-Cent(H)|.$$

\item \textit{$Z(H) = 1$ and $Z(K) \neq 1$}. A similar argument as Part (2) show that
 $$|2-Cent(H \times K)| = |2-Cent(H)||2-Cent(K)| + |2-Cent(K)|.$$

\item  \textit{$Z(H) = 1$ and $Z(K) = 1$}. Again as Part (2), $|2-Cent(H \times K)| = |2-Cent(H)||2-Cent(K)| + |2-Cent(H)| + |2-Cent(K)|.$
\end{enumerate}
Therefore,
$$|2-Cent(H \times K)| = |2-Cent(H)||2-Cent(K)| + \delta(K)|2-Cent(H)| + \delta(H)|2-Cent(K)|,$$
which completes the proof.
\end{proof}

\begin{exa}\label{exa2.5}
In this example the number of $2-$element subset centralizers of the group $S_n \times S_n$ is computed. Suppose $S_n$ has exactly $r$ $2-$element subset centralizers. Since $S_n$ is centerless,  $$|2-Cent(S_n \times S_n)| = |2-Cent(S_n)|^2 + 2|2-Cent(S_n)| = r^2 + 2r.$$
\end{exa}
\begin{cor} \label{2.6}
Suppose $H_1, H_2, \ldots, H_n$ are groups and $A=\{1,2,\ldots,n\}$. Then,
\begin{eqnarray*}
|2-Cent(\prod_{i=1}^{n}H_i)| & = & \prod_{i=1}^{n}|2-Cent(H_i)|\\
  & + & \sum_{i=1}^{n-1}\Biggl(\sum_{B_i\subset A}\biggl(\Bigl(\prod_{x\in B_i}\delta(H_x)\Bigr)\Bigl(\prod_{y\in B'_i}|2-Cent(H_y)|\Bigr)\biggr)\Biggr)
\end{eqnarray*}
in which $B_i$ is a $i-$subset of $A$ and $B'_i=A-B_i$. Moreover, by means of $\sum_{i=1}^{n-1}\sum_{B_i\subset A}$ we consider summation on all $i-$element subsets of $A$ such that $1\leq i\leq n-1$.
\end{cor}

\begin{proof}
Induct on $n$. By Theorem \ref{thm2.4}, the result is valid for $n = 2$. Suppose the result is correct for $n=k$, $k \geq 3$. Then,
\[|2-Cent(\prod_{i=1}^{k}H_i)| = \prod_{i=1}^{k}|2-Cent(H_i)|+M_{k-1}\]
in which
\[M_{k-1}=\sum_{i=1}^{k-1}\Biggl(\sum_{B_i\subset A}\biggl[\Bigl(\prod_{i\in B_i}\delta(H_i)\Bigr)\Bigl(\prod_{i\in B'_i}|2-Cent(H_i)|\Bigr)\biggr]\Biggr).\]
We prove the result for $n=k+1$. Since  $\delta(\prod_{i=1}^nH_i)=\prod_{i=1}^n\delta(H_i)$, by Theorem \ref{thm2.4}:
\begin{eqnarray*}
|2-Cent(\prod_{i=1}^{k+1}H_i)| & = & |2-Cent(\prod_{i=1}^{k}H_i\times H_{k+1})| \\
 & = & |2-Cent(\prod_{i=1}^{k}H_i)||2-Cent(H_{k+1})| \\
 &  & +\delta(H_{k+1})|2-Cent(\prod_{i=1}^{k}H_i)|+\delta(\prod_{i=1}^kH_i)|2-Cent(H_{k+1})| \\
 & = & (\prod_{i=1}^{k}|2-Cent(H_i)|+M_{k-1})|2-Cent(H_{k+1})| \\
 &  & +\delta(H_{k+1})(\prod_{i=1}^{k}|2-Cent(H_i)|+M_{k-1}) \\
 &  & +(\prod_{i=1}^n\delta(H_i))|2-Cent(H_{k+1})| \\
 & = & (\prod_{i=1}^{k}|2-Cent(H_i)|)|2-Cent(H_{k+1})| \\
 &  & +M_{k-1}|2-Cent(H_{k+1})|+\delta(H_{k+1})(\prod_{i=1}^{k}|2-Cent(H_i)|) \\
 &  & +\delta(H_{k+1})M_{k-1}+(\prod_{i=1}^n\delta(H_i))|2-Cent(H_{k+1})| \\
 & = & \prod_{i=1}^{k+1}|2-Cent(H_i)|+M_k,
\end{eqnarray*}
proving the corollary.
\end{proof}

\begin{exa}
In this example the case of  $n=3$ in Corollary \ref{2.6} is completed. Suppose $A=\{1,2,3\}$ and for $i=1,2$, $B_1=\{1\}, \{2\}$ or $\{3\}$ and $B_2 = \{1,2\}, \{1,3\}$ or $\{2,3\}$. Therefore,
\begin{eqnarray*}
|2-Cent(H_1\times H_2\times H_3)| & = & |2-Cent(H_1)||2-Cent(H_2)||2-Cent(H_3)|\\
  & +& \delta(H_1)|2-Cent(H_2)||2-Cent(H_3)|\\
  &+ & \delta(H_2)|2-Cent(H_1)||2-Cent(H_3)|\\
  & +& \delta(H_3)|2-Cent(H_1)||2-Cent(H_2)|\\
  &+ & \delta(H_1)\delta(H_2)|2-Cent(H_3)|\\
  & +& \delta(H_1)\delta(H_3)||2-Cent(H_2)|\\
  & +& \delta(H_2)\delta(H_3)|2-Cent(H_1)|.
\end{eqnarray*}
\end{exa}

\section{The Number of $2-$Element Centralizers of $CA-$Groups}

The aim of this section is to compute the number of $2-$element  centralizers of  $CA-$groups. We start by the following crucial result.

\begin{thm}\label{thm2.6}
Let $G$ be a $CA-$group. Then the following holds:
\begin{enumerate}
\item If $Z(G) = 1$, then $|2-Cent(G)| = |Cent(G)|$,
\item If $Z(G) \neq 1$, then $|2-Cent(G)| = |Cent(G)|+1$.
\end{enumerate}
\end{thm}

\begin{proof}
Since $G$ is a $CA-$group,  by Theorem \ref{thm1.2}(\ref{thm1.2(12)}), for every $x, y \in G \setminus Z(G)$, $C_G(x) = C_G(y)$ or $C_G(x) \cap C_G(y) = Z(G)$. Suppose $|Cent(G)|=r$. Then there are $x_1, \ldots, x_{r-1} \in G\setminus Z(G)$ such that $Cent(G) = \{G, C_G(x_1), \ldots, C_G(x_{r-1})\}$. Consider the following  two cases:
\begin{enumerate}
\item $Z(G) = 1$. It is clear that for every $1 \leq i \leq r-1$, $C_G(\{x_i, e\}) = C_G(x_i)$. By above argument  $C_G(\{x_i, x_j\}) = C_G(x_i) \cap C_G(x_j) = Z(G)$, where $1 \leq i,j \leq r-1$ and $i \neq j$. Furthermore,  there is no $x, y \in G$ such that $x \neq y$ and $G = C_G(x) \cap C_G(y)$. Therefore, $2-Cent(G) = \{Z(G), C_G(x_1), \ldots, C_G(x_{r-1})\}$ which implies that  $|2-Cent(G)|= |Cent(G)|$.

\item $Z(G) \neq 1$. It is clear that for every $x \in Z(G)$, $C_G(\{x,e\}) = C_G(x) = G$. By a similar argument as Part (1),  $2-Cent(G)$ = $\{G, Z(G), C_G(x_1), \ldots, C_G(x_{r-1})\}$. Thus, $|2-Cent(G)| = |Cent(G)| +1$.
\end{enumerate}
Hence the result.
\end{proof}

\begin{rem}\label{rem2.7}
Let $G$ be a centerless $CA-$group. Then  $|Cent(G)|=|2-Cent(G)|$ and so  $G$ is primitive $n-$ and $(2,n)-$centralizer.  Furthermore, if $G$  and $\frac{G}{Z}$ are both $CA-$group and $Z_2(G)\neq Z(G)$, then $G$ is primitive $n-$centralizer if and only if  it is primitive $(2,n)-$centralizer.
\end{rem}

\begin{cor}\label{cor2.8}
Let $G$ be a $CA-$group. $G$ is $(2,n)-$centralizer if and only if $G$ satisfies one of the following conditions:
\begin{enumerate}
\item $Z(G) = 1$ and $G$ is a $n-$centralizer.
\item $Z(G) \neq 1$ and $G$ is a $(n-1)-$centralizer.
\end{enumerate}
\end{cor}

\begin{proof}
The result follows from Theorem \ref{thm2.6}.
\end{proof}

\begin{thm}\label{thm2.9}
Let $H$ be an abelian group,  $p$ be a prime and $G = H \rtimes \mathbb{Z}_p$ is non-abelian. Then $|Z(G)| \mid |H|$, $|Cent(G)|=\frac{|H|}{|Z(G)|}+2$ and
$$|2-Cent(G)|=
\begin{cases}
|H|+2 & Z(G)=1, \\
\frac{|H|}{|Z(G)|}+3 & Z(G)\neq 1.
\end{cases}$$
\end{thm}

\begin{proof}
Since $G$ has an abelian normal subgroup of prime index, by Theorem \ref{thm1.2}(\ref{thm1.2(5)}), $|Cent(G)|=|G'|+2$. On the other hand, by Theorem \ref{thm1.2}(\ref{thm1.2(10)}), $|G|=p|G'||Z(G)|$ and  $|G'|=\frac{|H|}{|Z(G)|}$. Hence, $|Cent(G)|=\frac{|H|}{|Z(G)|}+2$. Moreover, since  $G$ has an abelian normal subgroup of prime index, by Theorem \ref{thm1.2}(\ref{thm1.2(11)}), $G$ is a $CA-$group. We now apply Theorem \ref{thm2.6} to deduce that
\[|2-Cent(G)|=
\begin{cases}
|H|+2 & Z(G)=1, \\
\frac{|H|}{|Z(G)|}+3 & Z(G)\neq 1,
\end{cases}\]
that completes the proof.
\end{proof}

\begin{cor}\label{cor2.10}
Suppose $\frac{G}{Z(G)}\cong \mathbb{Z}_n\rtimes \mathbb{Z}_p$ is non-abelian, $n$ is a positive integer and  $p$ is a prime number. Then,
\begin{enumerate}
\item If $Z(G)=1$, then $|Cent(\mathbb{Z}_n\rtimes \mathbb{Z}_p)|=|2-Cent(\mathbb{Z}_n\rtimes \mathbb{Z}_p)|=n+2$.

\item Suppose $Z(G)\neq1$.
\begin{enumerate}
\item If $Z_2(G)=Z(G)$ then $|2-Cent(G)|-1 = |Cent(G)| = |Cent(\frac{G}{Z(G)})| = |2-Cent(\frac{G}{Z(G)})| =  n+2$.

\item If $Z_2(G)\neq Z(G)$ then $|2-Cent(G)|-1 = |Cent(G)|=n+2$ and $|2-Cent(\frac{G}{Z(G)})|-1 = |Cent(\frac{G}{Z(G)})|=\frac{n}{|Z(\frac{G}{Z(G)})|}+2$.
\end{enumerate}
\end{enumerate}
\end{cor}

\begin{proof}
Since  $\frac{G}{Z(G)}\cong \mathbb{Z}_n\rtimes \mathbb{Z}_p$, it has an abelian normal subgroup of prime index and so by Theorem \ref{thm1.2}(\ref{thm1.2(11)}), $\frac{G}{Z(G)}$ is a $CA-$group. Since $\frac{G}{Z(G)}\cong \mathbb{Z}_n\rtimes \mathbb{Z}_p$ is non-abelian and  by Theorem \ref{thm1.2}(\ref{thm1.2(8)}),  $G$ is a $CA-$group.
\begin{enumerate}
\item Suppose $Z(G)=1$. By Remark \ref{rem2.7}, $|Cent(\mathbb{Z}_n\rtimes \mathbb{Z}_p)| = |2-Cent(\mathbb{Z}_n\rtimes \mathbb{Z}_p)|$ and by Theorem \ref{thm1.2}(\ref{thm1.2(7)}), $|Cent(\mathbb{Z}_n\rtimes \mathbb{Z}_p)| = n+2$. Hence the result.

\item Suppose $Z(G) \neq 1$.
\begin{enumerate}
\item If $Z_2(G) = Z(G)$ then $Z(\frac{G}{Z(G)}) = 1$. So by Theorems \ref{thm2.6}(1) and \ref{thm2.9}, $|Cent(\frac{G}{Z(G)})| = |2-Cent(\frac{G}{Z(G)})| = n+2$ and by Theorems \ref{thm2.6}(2) and \ref{thm1.2}(\ref{thm1.2(7)}), $|2-Cent(G)|-1 = |Cent(G)| = n+2$.

\item If $Z_2(G) \neq Z(G)$, then $Z(\frac{G}{Z(G)}) \neq 1$ and by Theorems \ref{thm2.6}(2) and \ref{thm1.2}(\ref{thm1.2(7)}), $|2-Cent(G)|-1 = |Cent(G)|=n+2$. Now by Theorems \ref{thm2.6}(2) and \ref{thm2.9}, $|2-Cent(\mathbb{Z}_n\rtimes \mathbb{Z}_p)|-1 = |Cent(\mathbb{Z}_n\rtimes \mathbb{Z}_p)|=\frac{n}{|Z(\mathbb{Z}_n\rtimes \mathbb{Z}_p)|}+2$.
\end{enumerate}
\end{enumerate}
This completes the proof.
\end{proof}

\begin{cor}\label{cor2.11}
Let $G$ be a group such that $Z(G)\neq1$, $\frac{G}{Z(G)} \cong \mathbb{Z}_n\rtimes \mathbb{Z}_p$ is non-abelian, $n$ is a positive integer and $p$ is prime. Then $G$ is primitive $n-$centralizer if and only if $Z_2(G)=Z(G)$.
\end{cor}

\begin{proof}
If  $Z_2(G)=Z(G)$, then by Corollary \ref{cor2.10}(2a), $G$ is  primitive $n-$centralizer. Conversely we assume that $G$ is  primitive $n-$centralizer. Then $|Cent(G)|=|Cent(\frac{G}{Z(G)})|$. On the other hand, by Theorems \ref{thm2.9} and \ref{thm1.2}(\ref{thm1.2(7)}), $n+2=\frac{n}{|Z(\frac{G}{Z(G)})|}+2$. Thus $|Z(\frac{G}{Z(G)})|=1$ and so $Z_2(G) = Z(G)$.
\end{proof}

\begin{cor}\label{cor2.12}
Let $G$ be a group with $\frac{G}{Z(G)} \cong \mathbb{Z}_n\rtimes \mathbb{Z}_p$ is non-abelian, $n$ is a positive integer and $p$ is prime. Then $G$ is primitive $(2,n)-$centralizer if and only if $Z(G)=1$.
\end{cor}

\begin{proof}
If $Z(G)=1$, then obviously $G$ is primitive $(2,n)-$centralizer. Conversely, we assume that $G$ is a primitive $(2,n)-$centralizer. Hence $|2-Cent(G)|=|2-Cent(\frac{G}{Z(G)})|$.  Suppose $Z(G)\neq 1$. Then by Corollary \ref{cor2.10}, both $G$ and $\frac{G}{Z}$ are $CA-$group.  We now consider the following two cases:
\begin{enumerate}
\item $Z(\frac{G}{Z(G)})=1$. By Theorem \ref{thm2.6} and Corollary \ref{cor2.10}(2a), $|2-Cent(\frac{G}{Z(G)})| = |Cent(\frac{G}{Z(G)})| = n+2$. Since $Z(G) \neq 1$,  by Theorems \ref{thm2.6} and \ref{thm1.2}(\ref{thm1.2(7)}), $|2-Cent(G)|=|Cent(G)|+1 = n+3$ and by our assumption $|2-Cent(G)|=|2-Cent(\frac{G}{Z(G)})|$. Thus, $n+3=n+2$, which is impossible.

\item $Z(\frac{G}{Z(G)}) \neq 1$. By Theorem \ref{thm2.6} and Corollary \ref{cor2.10}(2b), $|2-Cent(\frac{G}{Z(G)})| -1 = |Cent(\frac{G}{Z(G)})| = \frac{n}{|Z(\frac{G}{Z(G)})|}+2$. Since $Z(G) \neq 1$,  by Theorems \ref{thm2.6} and \ref{thm1.2}(\ref{thm1.2(7)}), $|2-Cent(G)|=|Cent(G)|+1 = n+3$. Now by our assumption, $|2-Cent(G)|=|2-Cent(\frac{G}{Z(G)})|$. Thus, $n+3=\frac{n}{|Z(\frac{G}{Z(G)})|}+3$ and so $Z(\frac{G}{Z(G)})=1$ which is impossible.
\end{enumerate}
Hence $G$ is centerless and the proof is complete.
\end{proof}

\begin{thm}\label{thm2.13}
Suppose $G$ is a finite non-abelian group and $r$ is the maximum size of a set of mutually non-commuting elements in $G$. Then $G$ is a $CA-$group if and only if
\begin{eqnarray*}
|2-Cent(G)| = \left\lbrace
\begin{array}{ll}
r+1 & Z(G) = 1 \\
r+2 & Z(G) \neq 1
\end{array}
\right..
\end{eqnarray*}
\end{thm}

\begin{proof}
Suppose $G$ is a $CA-$group. Then by Theorem \ref{thm1.2}(\ref{thm1.2(3)}), $|Cent(G)| = r+1$. On the other hand, by Theorem \ref{thm2.6},
\begin{eqnarray*}
|2-Cent(G)| = \left\lbrace
\begin{array}{ll}
|Cent(G)| & Z(G) = 1 \\
|Cent(G)| + 1 & Z(G) \neq 1
\end{array}
\right.,
\end{eqnarray*}
and hence,
 \begin{eqnarray*}
|2-Cent(G)| = \left\lbrace
\begin{array}{ll}
r+1 & Z(G) = 1 \\
r+2 & Z(G) \neq 1
\end{array}
\right..
\end{eqnarray*}

Conversely, we assume that
\begin{eqnarray*}
|2-Cent(G)| = \left\lbrace
\begin{array}{ll}
r+1 & Z(G) = 1 \\
r+2 & Z(G) \neq 1
\end{array}
\right..
\end{eqnarray*}
By Theorem \ref{thm1.2}(\ref{thm1.2(1)}), $r +1 \leq |Cent(G)|$. We now consider the following cases:
\begin{enumerate}
\item  $Z(G) = 1$. By Lemma \ref{lemm2.2}(2) and our assumption $|Cent(G)| \leq |2-Cent(G)| = r+1$. Thus $|Cent(G)| = r+1$ and by Theorem \ref{thm1.2}(\ref{thm1.2(3)}), $G$ is a $CA-$group.

\item  $Z(G) \neq 1$. By Lemma \ref{lemm2.2}(1) and our assumption $|Cent(G)| < |2-Cent(G)| = r+2$. Hence $|Cent(G)| = r+1$ and by Theorem \ref{thm1.2}(\ref{thm1.2(3)}), $G$ is a $CA-$group.
\end{enumerate}
This completes our argument.
\end{proof}

\section{Groups with at most Nine $2-$Element Centralizers}

The aim of this section is to characterize finite groups with at most nine $2-$element centralizers. One can easily seen that a group $G$ is $(2,1)-$centralizer if and only if $G$ is abelian which is similar to the case of $n-$centralizer groups. Also, there is no $(2,2)-$ and $(2,3)-$centralizer groups. In what follows it is also proved that there is no $(2,4)-$centralizer groups.

\begin{thm}\label{thm2.14}
There is no $(2,4)-$centralizer groups.
\end{thm}

\begin{proof}
Suppose $G$ is a $(2,4)-$centralizer group. We will consider two cases as follows:
\begin{enumerate}
\item \textit{$G$ is not centerless}. By Lemma \ref{lemm2.2}(1), $|Cent(G)| < |2-Cent(G)| = 4$. So $|Cent(G)| \leq 3$. On the other hand always $|Cent(G)| > 3$. This is impossible.

\item \textit{$G$ is centerless}. By Lemma \ref{lemm2.2}(2), $|Cent(G)| \leq |2-Cent(G)| = 4$ and since $|Cent(G)| > 3$, $|Cent(G)| = 4$. Then by Theorem \ref{thm1.1}(1), $G \cong \frac{G}{Z(G)} \cong \mathbb{Z}_2 \times \mathbb{Z}_2$, a contradiction.
\end{enumerate}
Hence there is no $(2,4)-$centralizer groups.
\end{proof}

\begin{lem}\label{lemm2.15}
If $|Cent(G)| = 6$ and $\frac{G}{Z(G)} \cong \mathbb{Z}_2 \times \mathbb{Z}_2 \times \mathbb{Z}_2 \times \mathbb{Z}_2$. Then $G$ is a $CA-$group.
\end{lem}

\begin{proof}
Suppose $\{x_1,\ldots, x_r\}$ is a set of pairwise non-abelian elements of $G$ with maximal size. So, by Theorem \ref{thm1.2}(\ref{thm1.2(1b)}), $r+1 \leq |Cent(G)| = 6$. Thus $r \leq 5$. On the other hand, by Theorem \ref{thm1.2}(\ref{thm1.2(1a)}), $r \geq 3$. This shows that $3 \leq r \leq 5$. If $r=3$, then by \ref{thm1.2}(\ref{thm1.2(1c)}) we have $|Cent(G)|=4$, which is a contradiction. If $r=4$, then by \ref{thm1.2}(\ref{thm1.2(1d)}), $|Cent(G)| = 5$ which leads to another contradiction. Therefore, $r=5$ and $Cent(G) = \{G, C_G(x_1), \ldots, C_G(x_5)\}$. Since $6 = |Cent(G)| < r+4 = 9$, by Theorem \ref{thm1.2}(\ref{thm1.2(2a)}) we have $C_G(x)$ is abelian, where $x \in G\setminus Z(G)$ is arbitrary. Thus $G$ is a $CA-$group, as desired.
\end{proof}
\begin{cor}\label{thm2.16}
Suppose $G$ is an $n-$centralizer with $n \leq 9$. Then $G$ is a $CA-$group.
\end{cor}

\begin{proof}
The proof follows from Theorems \ref{thm1.1}, \ref{thm1.2}(\ref{thm1.2(4)}) and Lemma \ref{lemm2.15}.
\end{proof}
\begin{lem}\label{thm2.17}
Suppose $G$ is a $(2,n)-$centralizer group with  $n \leq 9$. Then $G$ is a $CA-$group.
\end{lem}

\begin{proof}
Let $G$ be a $(2,n)-$centralizer and $n \leq 9$. If $Z(G) \neq 1$, then by Lemma \ref{lemm2.2}(1), $|Cent(G)| < |2-Cent(G)| = n \leq 9$. So by Corollary \ref{thm2.16}, $G$ is a $CA-$group. If $Z(G) =1$ then by Lemma \ref{lemm2.2}(2), $|Cent(G)| \leq |2-Cent(G)| = n \leq 9$. So by Corollary \ref{thm2.16}, $G$ is a $CA-$group.
\end{proof}


\begin{thm}\label{thm2.18}
A group $G$ is $(2,5)-$centralizer if and only if $G \cong S_3$ or $G$ is not centerless and $\frac{G}{Z(G)} \cong \mathbb{Z}_2 \times \mathbb{Z}_2$. Moreover, $G$ is primitive $(2,5)-$centralizer if and only if $G \cong S_3$.
\end{thm}

\begin{proof}
Suppose $G$ is $(2,5)-$centralizer. By Lemma \ref{thm2.17}, $G$ is a $CA-$group and by Corollary \ref{cor2.8}, the following two cases can be occurred:
\begin{enumerate}
\item \textit{$Z(G) = 1$ and $G$ is $5-$centralizer}. By Theorem \ref{thm1.1}(2), $G \cong \frac{G}{Z(G)} \cong \mathbb{Z}_3 \times \mathbb{Z}_3$ or $S_3$ and since $G$ is non-abelian, $G \cong S_3$.

\item \textit{$Z(G) \neq 1$ and $G$ is $4-$centralizer}. By Theorem \ref{thm1.1}(1), $\frac{G}{Z(G)} \cong \mathbb{Z}_2 \times \mathbb{Z}_2$, as desired.
\end{enumerate}
In order to prove the converse of this theorem, we note that $S_3$ is obviously $(2,5)-$centralizer. We assume that $\frac{G}{Z(G)} \cong \mathbb{Z}_2 \times \mathbb{Z}_2$.  Clearly $G$ can not be centerless and by Theorem \ref{thm1.1}(1), $|Cent(G)| = 4$. By Theorem \ref{thm1.2}(\ref{thm1.2(4)}), $G$ is a $CA-$group and by Theorem \ref{thm2.6}, $|2-Cent(G)| = |Cent(G)| + 1 = 5$. This proves that $G$ is a $(2,5)-$centralizer group.

If $G$ is primitive $(2,5)-$centralizer, then $|2-Cent(G)| = |2-Cent(\frac{G}{Z(G)})| = 5$. By the first part of this theorem, $G \cong S_3$ or $\frac{G}{Z(G)} \cong \mathbb{Z}_2 \times \mathbb{Z}_2$. Since $\mathbb{Z}_2 \times \mathbb{Z}_2$ is  abelian, $|2-Cent(\frac{G}{Z(G)})| = 1$ which is a contradiction. Thus $G \cong S_3$, as desired.
\end{proof}

\begin{thm}\label{thm2.19}
A group $G$ is $(2,6)-$centralizer if and only if $G \cong A_4$ or $G$ is not centerless and $\frac{G}{Z(G)} \cong \mathbb{Z}_3 \times \mathbb{Z}_3$ or $S_3$. Moreover, $G$ is primitive $(2,6)-$centralizer if and only if $G \cong A_4$.
\end{thm}

\begin{proof}
Suppose $G$ is $(2,6)-$centralizer. By Lemma \ref{thm2.17}, $G$ is a $CA-$group and by Corollary \ref{cor2.8}, the following two cases can be occurred:
\begin{enumerate}
\item \textit{$G$ is centerless and $G$ is a $6-$centralizer}. By Theorem \ref{thm1.1}(3), $G \cong \frac{G}{Z(G)} \cong \mathbb{Z}_2 \times \mathbb{Z}_2 \times \mathbb{Z}_2, \mathbb{Z}_2 \times \mathbb{Z}_2 \times \mathbb{Z}_2 \times \mathbb{Z}_2$, $D_8$ or $A_4$. Since $G$ is a non-abelian group and $|2-Cent(D_8)| = 5$,  $G \cong A_4$.

\item \textit{$G$ is not centerless and $G$ is $5-$centralizer}. By Theorem \ref{thm1.1}(2), $\frac{G}{Z(G)} \cong \mathbb{Z}_3 \times \mathbb{Z}_3$ or $S_3$, as desired.
\end{enumerate}

Conversely, if $G \cong A_4$ then $|2-Cent(A_4)|=6$. It is enough to assume that $G$ is not centerless and $\frac{G}{Z(G)} \cong \mathbb{Z}_3 \times \mathbb{Z}_3$ or $S_3$. By Theorem \ref{thm1.1}(2), $|Cent(G)|=5$ and  by Theorem \ref{thm1.2}(\ref{thm1.2(4)}), $G$ is a $CA-$group. Moreover,  by Theorem \ref{thm2.6}(2), $|2-Cent(G)| = |Cent(G)|+1= 6$. This proves that $G$ is $(2,6)-$centralizer.

If $G$ is a primitive $(2,6)-$centralizer, then $|2-Cent(G)| = |2-Cent(\frac{G}{Z(G)})| = 6$. By the first part of this theorem, $G \cong A_4$ or $G$ is not centerless and so $\frac{G}{Z(G)} \cong \mathbb{Z}_3 \times \mathbb{Z}_3$ or $S_3$. Since $|2-Cent(\mathbb{Z}_3 \times \mathbb{Z}_3)|=1$ and $|2-Cent(S_3)| = 5$, $G \cong A_4$ which completes our argument.
\end{proof}

\begin{thm}\label{thm2.20}
A group $G$ is $(2,7)-$centralizer if and only if $G \cong D_{10}, R$ or $G$ is not centerless and it is $6-$centralizer.
Moreover, $G$ is primitive $(2,7)-$centralizer if and only if $G \cong D_{10}$ or $R$.
\end{thm}

\begin{proof}
Suppose $G$ is $(2,7)-$centralizer. By Lemma \ref{thm2.17}, $G$ is a $CA-$group and by Corollary \ref{cor2.8}, $G$ is a  $6-$centralizer group with non-trivial center or a centerless  $7-$centralizer group. The first case leads to our result and in the second case, $G \cong D_{10}$ or $R$, as desired.

Conversely, if $G \cong D_{10}$ or $R$, then $|2-Cent(D_{10})|=|2-Cent(R)|=7$, as desired. So, it is enough to assume that $G$ is not centerless and $6-$centralizer. By Theorem \ref{thm1.1}(3), $\frac{G}{Z(G)} \cong \mathbb{Z}_2 \times \mathbb{Z}_2 \times \mathbb{Z}_2, \mathbb{Z}_2 \times \mathbb{Z}_2 \times \mathbb{Z}_2 \times \mathbb{Z}_2$, $D_8$ or $A_4$. By Theorem \ref{thm1.2}(\ref{thm1.2(4)}) and Lemma \ref{lemm2.15}, $G$ is $CA-$group and by Theorem \ref{thm2.6}(2), $|2-Cent(G)| = |Cent(G)|+1= 7$. This proves that $G$ is $(2,7)-$centralizer.

If $G$ is primitive $(2,7)-$centralizer. Then $|2-Cent(G)| = |2-Cent(\frac{G}{Z(G)})| = 7$. By the first part of this theorem, $G \cong D_{10}, R$ or $G$ is not centerless and $6-$centralizer. Since $|2-Cent(\mathbb{Z}_2 \times \mathbb{Z}_2\times \mathbb{Z}_2)|=|2-Cent(\mathbb{Z}_2 \times \mathbb{Z}_2\times \mathbb{Z}_2\times \mathbb{Z}_2)|=1$, $|2-Cent(D_8)| = 5$ and $|2-Cent(A_4)| = 6$. So $G \cong  D_{10}, R$ which completes our argument.
\end{proof}

\begin{thm}\label{thm2.21}
A group $G$ is $(2,8)-$centralizer if and only if $G$ is a $7-$centralizer group with non-trivial center.
Moreover, There is no primitive $(2,8)-$centralizer group.
\end{thm}

\begin{proof}
Suppose $G$ is $(2,8)-$centralizer. Apply Lemma \ref{thm2.17} to deduce that $G$ is a $CA-$group. By Corollary \ref{cor2.8}, $G$ is a centerless $8-$centralizer group or a $7-$centralize group with non-trivial center. If $G$ is a centerless $8-$centralizer group then by Theorem \ref{thm1.1}(6), $G \cong \frac{G}{Z(G)} \cong \mathbb{Z}_2 \times \mathbb{Z}_2\times \mathbb{Z}_2$, $D_{12}$ or $A_4$. But $G$ is a non-abelian group and  $|2-Cent(D_{12})|=|2-Cent(A_4)|=6$, which is impossible. Therefore, $G$ is a $7-$centralizer group with non-trivial center, as desired.

Conversely,  suppose $G$ is not centerless and it is a $7-$centralizer group. Then, by Theorem \ref{thm1.1}(5), $\frac{G}{Z(G)} \cong \mathbb{Z}_5 \times \mathbb{Z}_5$, $D_{10}$ or $R$, and by Theorem \ref{thm1.2}(\ref{thm1.2(4)}), $G$ is $CA-$group. We now apply Theorem \ref{thm2.6}(2) to deduce that $|2-Cent(G)| = |Cent(G)|+1= 8$. This proves that $G$ is $(2,8)-$centralizer.

If $G$ is primitive $(2,8)-$centralizer. Then $|2-Cent(G)| = |2-Cent(\frac{G}{Z(G)})| = 8$. By the first part of this theorem,  $G$ is a $7-$centralizer group with non-trivial center and by Theorem \ref{thm1.1}(5), $\frac{G}{Z(G)} \cong \mathbb{Z}_5 \times \mathbb{Z}_5$, $D_{10}$ or $R$. Since $|2-Cent(\mathbb{Z}_5 \times \mathbb{Z}_5)|=1$ and $|2-Cent(D_{10})| = |2-Cent(R)| = 7$. So, there is no primitive $(2,8)-$centralizer group which completes our argument.
\end{proof}

\begin{thm}\label{thm2.22}
A group $G$ is $(2,9)-$centralizer if and only if $G \cong D_{14}$, $Hol(\mathbb{Z}_7)$, a non-abelian group of order 21 or $G$ is a  $8-$centralizer group with non-trivial center. Moreover, $G$ is primitive $(2,9)-$centralizer if and only if $G \cong D_{14}$, $Hol(\mathbb{Z}_7)$ or a non-abelian group of order 21.
\end{thm}

\begin{proof}
Suppose $G$ is $(2,9)-$centralizer. Then, by Lemma \ref{thm2.17}, $G$ is a $CA-$group and by Corollary \ref{cor2.8}, $G$ is a $8-$centralizer group with non-trivial center or $G$ is a centerless $9-$centralizer group. In later, we apply Theorem \ref{thm1.1}(7) to deduce that $G \cong \frac{G}{Z(G)} \cong \mathbb{Z}_7 \times \mathbb{Z}_7$, $D_{14}$, $Hol(\mathbb{Z}_7)$ or a non-abelian group of order 21. Since $G$ is a non-abelian group,  $G \cong D_{14}$, $Hol(\mathbb{Z}_7)$ or a non-abelian group of order 21, as desired.

Conversely, if  $G \cong D_{14}$, $Hol(\mathbb{Z}_7)$ or a non-abelian group $L$ of order 21, then $|2-Cent(D_{14})| = |2-Cent(Hol(\mathbb{Z}_7))| = |2-Cent(L)| = 9$, as desired. So, it is enough to assume that $G$ is a $8-$centralizer group with non-trivial center. By Theorem \ref{thm1.1}(6),  $\frac{G}{Z(G)} \cong \mathbb{Z}_2 \times \mathbb{Z}_2 \times \mathbb{Z}_2$, $D_{12}$ or $A_4$, and  by Theorem \ref{thm1.2}(\ref{thm1.2(4)}), $G$ is $CA-$group. We now apply  Theorem \ref{thm2.6}(2) to deduce that $|2-Cent(G)| = |Cent(G)|+1= 9$. This proves that $G$ is $(2,9)-$centralizer.

If $G$ is primitive $(2,9)-$centralizer, then $|2-Cent(G)| = |2-Cent(\frac{G}{Z(G)})| = 9$. By the first part of theorem, $G \cong D_{14}$, $Hol(\mathbb{Z}_7)$, a non-abelian group of order $21$ or $G$ a $8-$centralizer group with non-trivial center. Since $|2-Cent(\mathbb{Z}_2 \times \mathbb{Z}_2\times \mathbb{Z}_2)|=1$ and $|2-Cent(D_{12})| =|2-Cent(A_4)| = 6$, $G \cong D_{14}$, $Hol(\mathbb{Z}_7)$ or a non-abelian group of order 21. This completes our argument.
\end{proof}

\section{Finite Groups with a Given Number of $2-$Eleemnt Centralizers}
In this section, a characterization of the alternating group $A_5$ with respect to the number of $2-$element centralizers is given. We also prove that all finite groups with  at most $21$ $2-$element centralizers are solvable.

\begin{lem}\label{lem3.1}
Let $G$ be a finite group such that $\frac{G}{Z(G)} \cong A_5$. Then $G \cong A_5$ or $|2-Cent(G)| = |Cent(G)| + 1 = 23, 33$.
\end{lem}

\begin{proof}
It is easy to see that  all element centralizers of $A_5$ are Sylow subgroups of $A_5$ and so each pair of them have trivial intersection. If $Z(G) = 1$ then $G \cong A_5$. Suppose $Z = Z(G) \ne 1$, then $G$ is not abelian and by Lemma \ref{lemm2.1}(2), $G \in 2-Cent(G)$. Choose arbitrary elements  $x, y \in G$. Obviously $\frac{C_G(x)}{Z} \leqslant C_{G/Z}(xZ)$ and $\frac{C_G(y)}{Z} \leqslant C_{G/Z}(yZ)$. Hence $\frac{C_G(x)}{Z} \cap \frac{C_G(y)}{Z} \leqslant C_{G/Z}(xZ) \cap C_{G/Z}(yZ) = 1_{G/Z}$. Therefore $\frac{C_G(x) \cap C_G(y)}{Z} = 1$ and so $C_G(x) \cap C_G(y) = Z$. Thus,  $Z \in 2-Cent(G)$ which  proves that $|2-Cent(G)| = |Cent(G)| + 1$. Finally,   Theorem \ref{thm1.2}(\ref{thm1.2(new)}) implies that $|2-Cent(G)| = 23$ or $33$.
\end{proof}

\begin{thm}
Let $G$ be a finite group. The following are hold:
\begin{enumerate}
\item  If $|2-Cent(G)| < 22$, then $G$ is solvable.
\item If $G$ is simple and $|2-Cent(G)| = 22$, then $G \cong A_5$.
\end{enumerate}
\end{thm}

\begin{proof}
Suppose $|2-Cent(G)| < 22$. By Lemma \ref{lemm2.2}, $|Cent(G)| \leq |2-Cent(G)| < 22$ and so by Theorem \ref{thm1.2}(\ref{thm1.2(15)}), $G$ is solvable. This proves part (1). We now assume that $G$ is simple and $|2-Cent(G)| = 22$. Again by Lemma \ref{lemm2.2}, $|Cent(G)| \leq 22$. If $|Cent(G)| \leq 21$, then by Theorem \ref{thm1.2}(\ref{thm1.2(15)}), $G$ is solvable, contradicts by simplicity of $G$. Therefore, $|Cent(G)| = 22$ and by Theorem \ref{thm1.2}(\ref{thm1.2(new1)}), $G \cong A_5$.
\end{proof}

\begin{rem}
Suppose $G$ is a finite non-abelian simple group with $|2-Cent(G)| \leq 100$. Then by Lemma \ref{lemm2.2}(2),  $|Cent(G)| \leq |2-Cent(G)| \leq 100$ and by \cite[Theorem A]{PSL}, the group $G$ is isomorphic to one of the simple  groups $PSL(2,5)$, $PSL(2,7)$ or $PSL(2,8)$.
\end{rem}

Our calculations with the aid of Gap suggest the following conjecture:

\begin{conj}
Suppose $G$ and $H$ are finite simple group and $|2-Cent(G)| = |2-Cent(H)|$. Then $G \cong H$.
\end{conj}

\begin{thm}
Let $G$ be a group with center $Z$ such that $[G:Z]=p^n$, $p$ is prime. Moreover, we assume that the order all proper centralizers of $G$ are equal to $p|Z|$. Then,
\begin{eqnarray*}
|Cent(G)| &=& p^{n-1} + p^{n-2} + \cdots + p + 2, \\
|2-Cent(G)| &=& |Cent(G)| + 1.
\end{eqnarray*}
\end{thm}

\begin{proof}
Since $G$ is non-abelian and $Z \ne 1$, Lemma \ref{lemm2.1}(2) implies that $G \in 2-Cent(G)$. We  claim that for each $x, y \in G \setminus Z$, $C_G(x) = C_G(y)$ or $C_G(x) \cap C_G(y) = Z$. To prove, we assume that  $C_G(x) \cap C_G(y) \neq Z$. Thus, $Z \lneqq C_G(x) \cap C_G(y) \leq C_G(x)$. Since $|C_G(x):Z| = p$, $C_G(x) \cap C_G(y) = C_G(x) = C_G(y)$. This shows that $C_G(x) = C_G(y)$. Hence $|2-Cent(G)| = |Cent(G)| + 1$. Suppose $m$ is  the number of distinct proper centralizer of $G$. Then $|G| - |Z| = m(|C_G(x)|-|Z|)$ and so $(p^n - 1)|Z| = m(p-1)|Z|$. Therefore, $m = \frac{p^n -1}{p-1} = p^{n-1} + p^{n-2} + \cdots + p + 1$. Thus $|Cent(G)| = p^{n-1} + p^{n-2} + \cdots + p + 2$. This completes the proof.
\end{proof}

\begin{thm}\label{thm5.7}
Let $p$ be a prime number and $G$ be a group with center $Z$ such that  $\frac{G}{Z}=\mathbb{Z}_p\times\cdots\times\mathbb{Z}_p$. If all proper centralizers of $G$ are of order $p|Z|$ or $p^2|Z|$. Then,
\begin{eqnarray*}
|Cent(G)| & = & s+t+1, \\
|2-Cent(G)| & = & |Cent(G)|+1,
\end{eqnarray*}
where $s$ and $t$ are the number of distinct centralizers of $G$ of orders $p|Z|$ and $p^2|Z|$, respectively. Moreover, $s+t(p+1)=p^{n-1}+p^{n-2}+\cdots +p+1.$
\end{thm}

\begin{proof}
Since $G$ is a non-abelian group and $Z$ is not a trivial subgroup, Lemma \ref{lemm2.1}(2) implies that $G \in 2-Cent(G)$. It is clear that for every $x\in G\setminus Z$, $x^p\in Z$. In what follow, two cases that $|C_G(x)|=p|Z|$ and $|C_G(x)|=p^2|Z|$ are considered separately.
\begin{enumerate}
\item $|C_G(x)|=p|Z|$. In this case, $A=\{x^iz\mid 0\leq i\leq p-1 \ \& \ z \in Z\}\subseteq C_G(x)$ and since $|A|=p|Z|$, $C_G(x)=A$. It is clear that  $C_G(x)\subseteq C_G(x^i)$, $1\leq i\leq p-1$. We prove that under the condition that $1\leq i\leq p-1$, $C_G(x) = C_G(x^i)$. Since $(i,p)=1$, there exists $m$ and $k$ such that $mi+kp=1$. Suppose $y\in C_G(x^i)$. Then, $y(x^i)^m = (x^i)^my$ and so $yx^{1-kp}  =  x^{1-kp}y$. Since $x^p \in Z$, $yx  =  xy$ which implies that $y \in C_G(x)$. Therefore,
\begin{equation}\label{equ1}
C_G(x^i)=C_G(x), 1\leq i\leq p-1.
\end{equation}

\item If $|C_G(x)|=p^2|Z|$. In this case, $A=\{x^iz\mid 0\leq i\leq p-1 \ \& \ z \in Z\}\subseteq C_G(x)$ and since $|A|=p|Z|$, $A\subset C_G(x)$. Choose $y\in C_G(x)\setminus A$. Then $B=\{x^iy^jz\mid 0\leq i,j\leq p-1\}\subseteq C_G(x)$.
Note that $|B|=p^2|Z|$ and so $C_G(x)=B$. Since $y\in C_G(x)$,  $C_G(x)\subseteq C_G(x^iy^j)$,  $1\leq i,j\leq p-1$. By assumption and our last inclusion, $p^2|Z|=|C_G(x)|\leq |C_G(x^iy^j)|\leq p^2|Z|$. Therefore,
\begin{equation}\label{equ2}
C_G(x^iy^j)=C_G(x), 1\leq i,j\leq p-1.
\end{equation}
\end{enumerate}
We claim that for every $x,y\in G\setminus Z$, one of the following hold:
\begin{enumerate}
\item[(1)] If $|C_G(x)|=|C_G(y)|$, then  $C_G(x)=C_G(y)$ or $C_G(x)\cap C_G(y)=Z$.

\item[(2)] If $|C_G(x)|\neq|C_G(y)|$, then $C_G(x)\cap C_G(y)=Z$.
\end{enumerate}
Suppose $C_G(x)\cap C_G(y)\neq Z$. Then there exists $u\in C_G(x)\cap C_G(y)\setminus Z$. Thus $u\in C_G(x)\setminus Z$ and $u\in C_G(y)\setminus Z$. We now apply Equations \ref{equ1} and \ref{equ2} to deduce that $C_G(x)=C_G(u)=C_G(y)$. This completes the proof of Parts $(1)$ and $(2)$.

Our above discussion show that $|2-Cent(G)|=|Cent(G)|+1$. Suppose the number of distinct centralizers of $G$ of orders $p|Z|$ and $p^2|Z|$ are $s$ and $t$, respectively. Therefore, $s(|C_G(x)|-|Z|)+t(|C_G(y)|-|Z|)  =  |G|-|Z|$ and so $s(p-1)|Z|+t(p-1)(p+1)|Z|  =  (p^n-1)|Z|$. This proves that $s+t(p+1)  =  p^{n-1}+p^{n-2}+\cdots +p+1$ which our argument.
\end{proof}

\section{Examples}
The aim of this section is to apply our results in Sections 2-5 for computing the number of $(2,n)-$centralizers in certain finite groups. We start by non-abelian $p-$groups of order $p^4$.

\begin{exa}
In this example we calculate the number of centralizers and $2-$element centralizers of a non-abelian $p-$group of order $p^4$.  It is proved that  $|Cent(G)|  =  p+2$, $p^2+2$ or $p^2+p+2$ and in any case $|2-Cent(G)|  =  |Cent(G)|+1$. Since $G$ is non-abelian, $|Z(G)|=p$ or $p^2$.
\begin{enumerate}
\item $|Z(G)|=p$. By Theorem \ref{thm1.2}(\ref{thm1.2(45)}), $|Cent(G)|=p^2+2$ or $p^2+p+2$.  Since $|G:Z(G)|=p^3$, by  Theorem \ref{thm1.2}(\ref{thm1.2(4)}), $G$ is a $CA-$group and by Theorem \ref{thm2.6}(2), $|2-Cent(G)|=|Cent(G)|+1$.

\item $|Z(G)|=p^2$. In this case,  $\frac{G}{Z(G)} \cong \mathbb{Z}_p\times \mathbb{Z}_p$ and by Theorem \ref{thm1.2}(\ref{thm1.2(9)}), $|Cent(G)|=p+2$. On the other hand, since $G$ is a group of order $p^4$ and $|Z(G)|=p^2$, $\frac{G}{Z(G)} \cong Z_p \times Z_p$. Therefore, by Theorem \ref{thm5.7}, $|2-Cent(G)|=|Cent(G)|+1$.
\end{enumerate}
\end{exa}

\begin{exa}
In this example, the number of $2-$element centralizers of a finite group $G$ with this property that $\frac{G}{Z(G)} \cong D_{2n}$ is computed, where $n \geq 3$ is a positive integer. By Corollary \ref{cor2.10}, one can easily seen that
\begin{enumerate}
\item If $Z(G)=1$, then $|Cent(D_{2n})|=|2-Cent(D_{2n})|=n+2.$
\item If $Z(G)\neq 1$ then,
\begin{enumerate}
\item If $n$ is odd. Then $|2-Cent(G)|-1=|Cent(G)|=|Cent(D_{2n})|=|2-Cent(D_{2n})|=n+2.$
\item If $n$ is even. Then $|2-Cent(G)|-1 = |Cent(G)|=n+2$ and $|2-Cent(D_{2n})|-1 = |Cent(D_{2n})|=\frac{n}{2}+2$.
\end{enumerate}
\end{enumerate}
\end{exa}

\begin{exa}
The semi-dihedral group $SD_{8n}$ can be presented as $\langle a, b \mid a^{4n} = b^{2} = e, bab = a^{2n - 1}\rangle$, where $n \geq 2$ is a positive integer. By Corollary \ref{cor2.10}, $$ |2-Cent(SD_{8n})| = |Cent(SD_{8n})| + 1 =\left\lbrace
\begin{array}{ll}
n+3 & n \ is \ odd\\
2n+3 & n \ is \ even
\end{array}
 \right.. $$
\end{exa}

\begin{exa}
The dicyclic group $T_{4n}$ can be presented as $\langle a, b \mid a^{2n} = e, a^n = b^{2}, b^{-1}ab = a^{-1}\rangle$, where $n \geq 2$ is a positive integer. Since $\frac{T_{4n}}{Z(T_{4n})} \cong D_{2n} = Z_n \rtimes Z_2$, by Corollary \ref{cor2.10}, we have  $|2-Cent(T_{4n})| = |Cent(T_{4n})| +1 = n + 3.$
\end{exa}

\begin{exa}
The  group $V_{8n}$ can be presented as $\langle a, b \mid a^{2n} = b^4 = e, aba = b^{-1}, ab^{-1}a = b\rangle$, where $n$ is a positive integer. Note that $$\frac{V_{8n}}{Z} \cong \left\{ \begin{array}{ll} Z_{2n} \rtimes Z_2 & 2\nmid n\\ Z_n \rtimes Z_2  & 2 | n \end{array}\right..  $$
Then by Corollary \ref{cor2.10}, we have
$$ |2-Cent(V_{8n})| = |Cent(V_{8n})| + 1 = \left\lbrace
\begin{array}{ll}
2n+3 & n \ is \ odd \\
n+3 & n \ is \ even
\end{array}
 \right..  $$
\end{exa}

\begin{exa}
The group $U_{2(n,m)}$ can be presented as $U_{2(n,m)} = \langle a, b \mid a^{2n} = b^m = e, aba^{-1} = b^{-1}\rangle$. 
If $m= 1, 2$. Then $U_{2(n,m)}$ is an abelian group and so $|2-Cent(U_{2(m,n)})| = |Cent(U_{2(m,n)})| = 1$. In other cases, $\frac{U_{2(m,n)}}{Z} \cong Z_m \rtimes Z_2$ and by Corollary \ref{cor2.10}, 
$$
|2-Cent(U_{2(m,n)})| =
\left\lbrace
\begin{array}{ll}
|Cent(U_{2(m,n)})| = m+2 & m \ is \ odd \ and \ n=1 \\
|Cent(U_{2(m,n)})| + 1 = m+3 & m \ is \ odd \ and \ n \neq 1 \\
|Cent(U_{2(m,n)})| + 1 = \frac{m}{2} + 3 &  m \ is \ even
\end{array}
\right..
$$
\end{exa}

\vskip 3mm

\noindent{\bf Acknowledgement.}   The research of the authors are partially supported by the University of Kashan under grant no
785149/70.

\end{document}